\documentclass[12pt]{amsart}

\usepackage{amsfonts,amsmath,latexsym,amssymb,verbatim,amsbsy,amsthm}
\usepackage{graphicx}
\usepackage{soul}
\usepackage{nicefrac}
\usepackage{bm}

\usepackage[top=1in, bottom=1in, left=1in, right=1in]{geometry}

\usepackage[dvipsnames]{xcolor}

\usepackage[colorlinks=true, pdfstartview=FitV, linkcolor=RoyalBlue,citecolor=ForestGreen, urlcolor=blue]{hyperref}

\numberwithin{equation}{section}

\renewcommand{\geq}{\geqslant}

\renewcommand{\leq}{\leqslant}

\theoremstyle{plain}
\newtheorem{THEOREM}{Theorem}[section]

\newtheorem{theorem}[THEOREM]{Theorem}
\newtheorem{corollary}[THEOREM]{Corollary}

\newtheorem{proposition}[THEOREM]{Proposition}

\theoremstyle{assume}
\newtheorem{assume}[THEOREM]{Assumption}

\theoremstyle{definition}

\newtheorem{definition}[THEOREM]{Definition}

\theoremstyle{remark}

\newtheorem{remark}[THEOREM]{Remark}

\def \a {\alpha} 

\def\bu{{\mathbf u}}

\def \bc {{\bf c}}
\def \bb {{\bf b}}

\def \bx {{\mathbf x}}
\def \by {{\mathbf y}}
\def \bu {{\mathbf u}}
\def \bv {{\mathbf v}}
\def \bw {{\mathbf w}}

\newcommand{\vc}[1]{{\bf #1}}
\newcommand{\vv}{\vc{\bv}}
\newcommand{\vx}{\vc{\bx}}

\renewcommand{\S}{\ensuremath{\mathbb{S}}} 

\def \drho {\, \mbox{\upshape{d}}\rho}
\def \dh  {\, \mbox{\upshape{d}}h}

\def \dt  {\, \mbox{\upshape{d}}t}
\def \dh  {\, \mbox{d}h}

\def \dy  {\, \mbox{\upshape{d}}y}

\def \ddt  {\frac{\mbox{d\,\,}}{\mbox{d}t}}

\def\bx{{\mathbf x}}

\def\XX{\mathbb X}

\usepackage{mathrsfs}
\usepackage{amsfonts}
\usepackage{amsfonts}\usepackage{mathrsfs}\usepackage{mathrsfs}\usepackage{graphicx}\usepackage{enumerate}\usepackage{epsfig}
\usepackage{subfigure}\usepackage{amsfonts}\usepackage{amscd}\usepackage{amsthm}\usepackage{bbm}
\usepackage{extpfeil}
\usepackage{colortbl}\usepackage{float}

\newcommand{\bbr}{\mathbb R}

\newcommand{\bbt}{\mathbb T}

\newcommand{\be}{\begin{equation}}
\newcommand{\ee}{\end{equation}}
\newcommand{\kphi}{\phi}
\newcommand{\kPhi}{\Phi}
\newcommand{\Rad}{R} 
\newcommand{\rc}{r} 
\newcommand{\Sp}{{\mathcal S}} 
\newcommand{\Spt}{{\Sp(t)}} 
\newcommand{\ind}{n} 
\newcommand{\JR}{|\delta \bu(t)|_{2,r}^2}

\begin{document}

\title[Flocking with short-range interactions]{Flocking with short-range interactions}

\author{Javier Morales}
\address{\newline
Center for Scientific Computation and Mathematical Modeling (CSCAMM)\newline
University of Maryland, College Park MD 20742}
\email{javierm1@cscamm.umd.edu}

\author{Jan Peszek}
\address{\newline
Center for Scientific Computation and Mathematical Modeling (CSCAMM)\newline
University of Maryland, College Park MD 20742}
\email{jpeszek@cscamm.umd.edu}

\author{Eitan Tadmor}
\address{\newline
Department of Mathematics\newline
Center for Scientific Computation and Mathematical Modeling (CSCAMM)\newline
 and\newline
  Institute for Physical Sciences \& Technology (IPST)\newline
 University of Maryland, College Park MD 20742}
\email{tadmor@cscamm.umd.edu}

\date{\today}



\thanks{\textbf{Acknowledgment.} Research was supported  by NSF grants 
	DMS16-13911, RNMS11-07444 (KI-Net) and ONR grant 	N00014-1812465.
JP was also supported by the Polish MNiSW grant Mobilno\' s\' c Plus no. 1617/MOB/V/2017/0.	
}


\begin{abstract} We study the large-time behavior of continuum alignment dynamics based on Cucker-Smale (CS)-type interactions which involve \emph{short-range}  kernels, that is, communication kernels with support much smaller than the diameter of the crowd.  We show that if the amplitude of the interactions is larger than a finite threshold, then unconditional hydrodynamic flocking follows. Since we do not impose any regularity nor do we   require  the kernels to be bounded, the result covers  both regular and singular interaction kernels.\newline
Moreover, we treat initial densities  in the general class of compactly supported measures which are required to have  positive  mass \emph{on average} (over balls at small enough scale), but otherwise vacuum is allowed at smaller scales. Consequently, our arguments of hydrodynamic flocking  apply, mutatis mutandis,  to  the agent-based CS model with  finitely many Dirac masses. 
In particular, discrete flocking threshold is shown to depend on the number of dense clusters  of communication but otherwise does not grow with the number of agents.
\end{abstract}

\maketitle
\setcounter{tocdepth}{1}
\vspace*{-1.0cm}
\tableofcontents

\section{Introduction and statement of main results}

Our starting point is the celebrated Cucker-Smale  (CS) agent-based model \cite{CS2007a,CS2007b} which governs the alignment dynamics of $N$ agents --- each identified by its (position, velocity) pair, $(\vx_i(t),\vv_i(t)) \in \Omega\times \bbr^d$,
\begin{equation}\label{cs}
\begin{cases}
\dot{\vx}_i =\vv_i,\\
\dot{\vv}_i =\displaystyle\frac{\kappa}{N}\displaystyle\sum_{j=1}^{N}\kphi(|\vx_i-\vx_j|)(\vv_j-\vv_i),
\end{cases}
\end{equation}
and is subject  to prescribed initial conditions $(\vx_i(0),\vv_i(0))=(\vx_{i0},\vv_{i0})$. Here we study the $d$-dimensional ambient space where $\Omega$ is either $\bbt^d$ or $\bbr^d$.

The large crowd dynamics, $N\to\infty$, is captured by the hydrodynamic description
\begin{equation}\label{system}\hspace{6em}
\begin{cases}
\partial_{t}\rho +\nabla_\bx\cdot (\rho \bu) =0,\\
\partial_{t}\bu+\bu\cdot\nabla_\bx \bu =\kappa{\displaystyle \int_{\Omega}\kphi(|\bx-\by|)\big(\bu(\by)-\bu(\bx)\big) \drho(\by),}
\end{cases}\hspace{1mm} 
\end{equation}
which governs  the density and velocity of the crowd, $(\rho(t,\bx),\bu(t,\bx))$ at $(t,\bx) \in (\bbr_+\times \Omega)$, subject to prescribed data $(\rho(0,\bx),\bu(0,\bx))=(\rho_0(\bx),\bu_0(\bx))$.

The main feature of both --- the agent-based model \eqref{cs} and its large crowd hydrodynamic description \eqref{system}, is the \emph{communication} or \emph{interaction kernel},  $\kphi:[0,\infty)\to[0,\infty)$, which controls pairwise interactions depending on the \emph{geometric distances} of agents at $(t,\bx)$ and $(t,\by)$. Here  $\kappa>0$ quantifies the  amplitude of the interaction  and we distinguish between two main classes of interactions, involving long-range and short-range kernels.

\subsection{Long-range interactions} The role of $\kphi$ in determining the large-time behavior  becomes evident once we consider the \emph{energy fluctuations} associated with \eqref{cs} or \eqref{system}. Let $\Sp(t)$ denote the support of $\rho(t,\cdot)$ and assume, without loss of generality, that $\rho_0$ has a unit mass, then conservation of mass implies
\be\label{eq:mass}
\int_{\Sp(t)} \drho(t,\cdot) \equiv 1, \qquad \Sp(t):=\text{supp}\, \rho(t,\cdot).
\ee  
The basic bookkeeping associated  with \eqref{system} quantifies  the decay rate of the $L^2$-energy \emph{fluctuations}\footnote{Here and below we use $\delta \cdot$ to denote \emph{fluctuations}, $\delta W\equiv (\delta W)(\bx,\by):=W(\bx)-W(\by)$ with the corresponding weighted norms taken on the product space $\Sp\times \Sp$,
e.g., 
$|\delta \bu|^2_2= \int |\bu(\bx)-\bu(\by)|^2 \drho(\bx)\drho(\by)$. Likewise, $(\delta \bv)_{ij}=\bv_i-\bv_j$ with $|\delta\bv|_\infty=\max_{i,j}|\bv_i-\bv_j|$ etc.}
\[
|\delta \bu(t)|_2^2 :=
\int_{\Sp(t)\times\Sp(t)}\hspace*{-0.5cm} |\bu(t,\bx)-\bu(t,\by)|^2\drho(\bx)\drho(\by),
\]
in terms of  the enstrophy,
\begin{align}\label{ens}
\ddt\frac{1}{2}|\delta \bu(t)|_2^2 = -\kappa\int_{\Sp(t)\times\Sp(t)}\hspace*{-0.5cm} \kphi(|\bx-\by|) |\bu(\bx)-\bu(\by)|^{2}\drho(\bx)\drho(\by).
\end{align}
The \emph{enstrophy} on the right  admits the obvious lower-bound
\[
\int_{\Sp(t)\times\Sp(t)}\kphi(|\bx-\by|) |\bu(\bx)-\bu(\by)|^{2}\drho(\bx)\drho(\by)
\geq \mathop{\min}_{h\leq \text{diam}(\Sp(t))} \hspace*{-0.2cm}\phi(h) \times |\delta \bu(t)|_2^2.
\]
Since the diameter of velocity field, $ \displaystyle\sup_{\bx,\by\in \Sp(t)}|\bu(t,\bx)-\bu(t,\by)|$, remains bounded, the support of $\rho$ cannot expand faster than 
\[
\text{diam}(\Sp(t)) \leq \text{diam}(\Sp_0)+|\delta \bu_0|_\infty\cdot t, \qquad |\delta \bu(t)|_\infty := \displaystyle\sup_{\bx,\by\in \Sp(t)}|\bu(t,\bx)-\bu(t,\by)|,
\]
and hence \eqref{ens} implies the unconditional \emph{flocking behavior} of \eqref{system}, in the sense that
\be\label{eq:flocking}
\int_0^\infty \phi(h)\dh=\infty  \ \ \leadsto \ \ |\delta \bu(t)|_2^2 \ \ \stackrel{t\rightarrow \infty}{\longrightarrow} \ 0,
\ee
holds for all initial data which give rise to smooth solutions. 
Thus, if $\kphi$ is a \emph{long-range} interaction kernel satisfying the `fat tail' condition \eqref{eq:flocking}, then `smooth solutions must flock', \cite{TT2014},
and the existence of smooth  solutions is known in $\Omega=\bbt^1$, \cite{DKRT2018, ST2017a, ST2017b, ST2017c}, in  $\Omega=\bbr^2$ \cite{HeT2017},  and with small data in $\Omega = \bbt^d$ and $\Omega = \bbr^d$ \cite{Sh2018, DMPW2018}. 

The situation is completely analogous to the agent-based  case supported on the discrete set $\Sp(t):=\{\vx_i(t)\}_{i=1}^N$, where  the decay of energy fluctuations is controlled by
\begin{align}\label{enidd}
\ddt|\delta \vv(t)|_2^2 \leq  -\kappa\cdot
 \hspace*{-0.3cm}\mathop{\min}_{h\leq \text{diam}(\Sp(t))} \hspace*{-0.2cm}\phi(h)\times  |\delta \vv(t)|_2^2, \qquad 
 |\delta \vv(t)|_2^2 :=\frac{1}{N^2} \displaystyle\sum_{i,j=1}^{N}|\vv_i(t)-\vv_j(t)|^2.
\end{align}
But since the diameter of velocities, $\displaystyle 
\mathop{\max}_{1\leq i,j\leq N}|\vv_i(t)-\vv_j(t)|$, remains bounded, the diameter of positions grows at most linearly in time
\[
\text{diam}(\Sp(t)) \leq \text{diam}(\Sp_0) + |\delta \vv(0)|_\infty\cdot t, 
\qquad |\delta \vv(t)|_\infty :=
\mathop{\max}_{1\leq i,j\leq N}|\vv_i(t)-\vv_j(t)|,
\]
 and we conclude that long-range communication kernels with fat-tail imply flocking,
 \cite{HT2008, HL2009, MT2014},
 \be\label{eq:dflocking}
\int^\infty \phi(h)\dh=\infty  \ \ \leadsto \ \ |\delta \bv(t)|_2^2 \ \ \stackrel{t\rightarrow \infty}{\longrightarrow} \ 0.
\ee

\subsection{Short-range interactions} From the perspective of practical applications, it is important to understand the large-time behavior of alignment dynamics driven by 
\emph{short range} kernels, specifically --- when $\text{diam}(\text{supp}\, \kphi)$ is strictly less than the support of the crowd,
$\text{diam}(\Sp(t))$.  The fascinating and non-trivial aspect is to understand when and how such \emph{short-range}  interactions lead to the emergence of large-scale patterns, which in our case are realized by the flocking behavior approaching the average\footnote{Recall that $\rho$ is a probability measure} $\overline{\bu}_0$,   
\[
\int_{\Omega} |\bu(t,\bx)-\overline{\bu}_0|^2 \drho(\bx)\longrightarrow 0, \qquad 
\overline{\bu}_0:=\int_{\Sp_0} \bu_0(\bx)\drho(\bx),
\]
and likewise $|\vv_i(t) -\overline{\vv}_0|\rightarrow 0$ in the discrete case.
The dynamics of such short-range interactions   lies outside the `fat-tale' regime  \eqref{eq:flocking},\eqref{eq:dflocking} and is the main focus of the present work.

We begin by  fixing an \emph{essential support} of $\kphi$ of size $\Rad$ where it is assumed to be bounded away from zero,   
\begin{align}\label{psi}
\kphi:(0,\infty)\to(0,\infty),\quad  \kphi(h)\chi_{{}_{[0,\Rad]}}(h) \geq 1 >0 \ \ \mbox{for}\ \ h\in [0,\Rad].
\end{align}
We emphasize that the only requirement made in \eqref{psi} is that the kernel has a finite-range, $\Rad$, of non-negligible interaction which is scaled here as $\kphi(h)\geq 1$, but otherwise, we make no reference to the behavior of $\kphi$ \emph{outside} $[0,\Rad]$. In particular, a smooth $\kphi$ may have a compact support, say,  ${\rm supp}(\kphi) = [0,2\Rad]$, or --- since we do not impose any regularity restrictions, $\kphi$ can be cut-off at $h=\Rad$. In either case, $\kphi$ is not restricted to maintain a \emph{direct} communication between every two agents in the crowd. We summarize by introducing the main spatial scales involved in the problem, namely
\begin{subequations}\label{eqs:Ds}
\begin{align}
& \mbox{The diameter of the initial `crowd' distribution}: \quad D_{\Sp_0}:=\sup_{\bx,\by \in \Sp_0}|\bx-\by|, \label{eq:DS}\\
& \mbox{The essential diameter of the communication}: \quad \ \  D_\kphi:=\sup\{\rc \ | \ \min_{h\leq \rc}\kphi(h)\geq 1\}. \label{eq:Dp}
\end{align}
\end{subequations}
We specifically address the case of short-range interactions $D_{\kphi} \ll D_{\Sp_0}$. Moreover, since we do not impose any boundedness of $\kphi$, \eqref{psi} includes both --- bounded  communication  kernels, \cite{CS2007a,CS2007b, HT2008, HL2009,CFTV2010, MT2014}, and singular ones \cite{Pe2014, Pe2015, PS2017, ST2017a,ST2017b,ST2017c,DKRT2018}.\newline
The flocking behavior in such cases hinges on the \emph{connectivity} of the communication throughout the crowd: in the discrete case, it is the \emph{graph connectivity} of the adjacency matrix $\{\phi(|\vx_i(t)-\vx_j(t)|)\}$, \cite[theorem 4.2]{MT2014}.  Here we show that in the continuum, it is the \emph{chain connectivity} at a communication  scale  that matters (see assumption \ref{ass:s0} below): if chain connectivity at the communication scale persists in time then the dynamics will converge to a flock. Our main results show that if the amplitude of alignment is strong enough, $\kappa>\kappa_0$, such chain connectivity persists and hence uniform flocking follows.
\begin{theorem}[{\bf Emergence of hydrodynamic flocking}]\label{flockcont}
Let $(\rho, \bu)$ be a strong solution of \eqref{system},\eqref{psi} in either $\bbt^d$ or $\bbr^d$ subject to compactly supported initial data $(\rho_0,\bu_0)\in (L^1(\Sp_0), C^1(\Sp_0))$. Assume  $\rho_0$ is normalized to have  a unit mass and it satisfies the following $($see assumption \ref{ass:s0} below$)$:

\smallskip\noindent
 (i) Its support is chain connected with $\ind(\rc)$ balls at scale $\rc$ for some $\rc \leq \frac{1}{6}D_\kphi$ $($see \eqref{cdef}$)$; 
 
 \smallskip\noindent
 (ii) It has a minimal average mass $m>0$ at  scale $\frac{\rc}{100}$ $($see \eqref{mdef}$)$.
 
\smallskip \noindent
Then, there exist  constants $\displaystyle \eta \lesssim \frac{m^2}{n}$  and a threshold $\kappa_0 =\kappa_0(n,m,r)$,
such that if the amplitude of the alignment $\kappa \geq \kappa_0$,  then there is exponential convergence towards  flocking,
\begin{align}\label{fc0}
|\bu(t,\cdot) -\overline{\bu}_0|_{\infty}\leq |\delta \bu_{0}|_{\infty}e^{ -\kappa  m t} + \frac{2}{m}|\delta \bu_{0}|_{2}e^{ -\kappa \eta t}.
\end{align}
\end{theorem}

\begin{remark}[{\bf On the size of the threshold} $\kappa_0$]
Noting that $nm \lesssim \rho(\Sp_0)=1$, it follows that the threshold $\kappa_0$  specified in \eqref{kll} below, does not exceed
\begin{subequations}\label{eqs:fc}
\begin{equation}\label{eq:fc1}
\kappa_0 \lesssim \max\{|\delta \bu_0|_\infty,|\delta \bu_0|_2\}\frac{n}{m^3 \rc } = C_0\times \frac{1}{m^4 \rc}, \qquad C_0=\max\{|\delta \bu_0|_\infty,|\delta \bu_0|_2\}.
\end{equation}
In the generic case that the initial support $\Sp_0$ is \emph{convex}, the length of a connecting chain at scale $\rc$ does not exceed the ratio $\displaystyle n \lesssim\frac{D_{\Sp_0}}{\rc}$ and  we end up with a slightly improved bound
\begin{equation}\label{eq:fc2}
\kappa_0   \lesssim C_0\times \frac{D_{\Sp_0}}{m^3 \rc^2}, \qquad r\lesssim \frac{1}{6}D_\kphi.
\end{equation}
\end{subequations}
\end{remark}

\begin{remark}[{\bf Local vs global communication}]
We note that if the communication kernel $\kphi$ has a finite support yet it maintains global communication at $t=0$, that is,  if $D_\kphi>D_{\Sp_0}$ then this  yields \emph{conditional} flocking provided  $\kappa$ is assumed to be large enough; this follows from tracing the decreasing Ha-Liu functional, \cite{HL2009}, 
\be\label{eq:HL}
|\delta \bu(t)|_\infty + \kappa \int_{D_{\Sp_0}}^{\text{diam}(\Sp(t))}\hspace*{-0.4cm}\phi(h)\dh
\leq |\delta \bu(0)|_\infty
\ee
 The focus of theorems \ref{flockcont} and \ref{flockpart}, however,  is in short-range communication where 
 \be\label{eq:short}
D_{\kphi} < D_{\Sp_0}.
 \ee 
 The necessity of  large enough $\kappa$ in such local cases is clear: otherwise,  it is easy to produce an example --- we do it in section \ref{remarks} below, where agents  can leave their sensitivity regions in a finite time and then move independently with constant velocities. Observe that a straightforward application of \eqref{eq:HL} to the re-scaled system \eqref{flockcont}, $(t, \bu) \mapsto (\kappa t, \bu/\kappa)$, would fall short to imply that $\text{diam}(\Sp(t))$ remains bounded  (and hence flocking) since $\kphi(h)$ vanishes for $h> D_{\Sp_0}$ in the case of short-range communication \eqref{eq:short}.
\end{remark}

\begin{remark}[{\bf Beyond almost aligned dynamics}]
The arguments behind theorem \ref{flockcont}  do \emph{not} require any smoothness beyond the existence of a classical flow map, $\{\XX(t,\bx)\, | \, \ddt\XX(t,\bx)=\bu(t,\XX(t,\bx))\}$ and $\rho_0$ being a non-negative measure, but otherwise, the  resulting  flocking estimate \eqref{fc0} does  not depend on  higher $L^{p}$ bound of $\rho_0$ nor on the higher regularity of $\bu$. Our result, therefore,  cannot be derived  by rescaling  the smoothness  arguments   developed for the almost aligned  regimes in \cite{Sh2018} and \cite{DMPW2018}. 
\end{remark}

\subsection{Alignment in discrete dynamics} The measure $\rho_0$ is only required to have a positive  mass \emph{on average} (over balls at scale $\geq \frac{\rc}{100}$ --- consult \eqref{mdef} below) but otherwise vacuum is allowed at smaller scales. Consequently, the arguments of theorem  \ref{flockcont}   apply, mutatis mutandis,  to deal with $\rho_0$ consisting of finitely many Dirac masses which are pushed forward along particle paths. This brings us to
our second main result, which addresses the flocking of   the discrete setup  \eqref{cs} with short-range communication.

\begin{theorem}[{\bf Flocking of CS model}]\label{flockpart}
Let $(\vx(t),\vv(t)), \ \vx=(\vx_1,\ldots, \vx_N),\vv=(\vv_1,\ldots,\vv_N)$, be a classical solution of  the CS agent-based system \eqref{cs},\eqref{psi} subject to the initial data $(\vx_0,\vv_0)$.  Assume that $\vx_0$ is chain connected at scale $\rc, \ \rc\lesssim \frac{1}{6}D_\kphi$.
Then, there exists  a constant $\kappa_0$
\begin{align}\label{eq:dthresh}
\kappa_0:=\frac{16N^4}{\rc}|\delta \vv_{0}|_{2}, \quad \rc\lesssim \frac{1}{6}D_\kphi,
\end{align}
such that for $\kappa \gtrsim \kappa_0$ the velocity fluctuation decays exponentially towards a flock, {\it i.e.}, we have 
\begin{align}\label{fc1}
|\vv(t,\cdot) -\overline{\vv}_0|_{\infty}\leq |\delta \vv_{0}|_{\infty}\,e^{ -\frac{\kappa }{N}  t} + 2N\cdot |\delta \vv_{0}|_{2}\,e^{ -\frac{\kappa }{N^3}  t}.
\end{align}
\end{theorem} 

\begin{remark}[{\bf Dynamics in discrete clusters}]\label{rem:moderate}
The discrete threshold \eqref{eq:dthresh} follows from theorem \ref{flockcont} with 
$n=N$ Dirac masses, each scaled with mass $m=\nicefrac{1}{N}$.
But we observe that the mass of a ball with radius $r$  centered at one of the $\vx_j(0)$'s may in fact be much larger than $\nicefrac{1}{N}$, if that ball encloses a crowd of Dirac masses. In this case the threshold bound of order $\kappa_0 \sim {n}{m^{-3}}$ improves with growing $m$. In particular, consider initial configurations where the $N$ agents can be partitioned  into clusters, $\{{\mathcal N}_\alpha\}_{\alpha}$ of size $n_\alpha$, which partition $\{1,2, \ldots, N\}$, such that the agents in each initial cluster are in direct communication
\[
{\mathcal C}_\alpha:=\{x_{i}(0)\ | \ i\in{\mathcal N}_\alpha\}, \qquad n_\alpha=\#{\mathcal N}_\alpha, \quad  \sum_{\alpha} n_\alpha=N.
\]
 Thus,  the crowd in cluster ${\mathcal C}_\alpha$ is inside a ball of one of its  members, say ${\mathcal C}_\alpha \subset B_{\rc}(x_{i_\alpha}(0))$ with $\rc \lesssim \frac{1}{6}D_\kphi$. The mass inside a fixed fraction of each such ball is of order $\displaystyle m \gtrsim \frac{n_\alpha}{N}$, and in particular,
 \[
 m\gtrsim \frac{n_-}{N}, \qquad n_-:=\min_\alpha n_\alpha.
 \]
The number of such clusters does not exceed  $n\leq \nicefrac{N}{n_-}$ and we end up with a threshold of order $\kappa_0 \gtrsim (\nicefrac{N}{n_-})^4$. We conclude the following.
 \end{remark}
 \begin{corollary}[{\bf Finitely many clusters}]\label{cor:clusters} Let $(\vx(t),\vv(t))$, be a classical solution of  the CS agent-based system \eqref{cs},\eqref{psi} subject to the initial data $(\vx_0,\vv_0)$ which can be partitioned  into clusters ${\mathcal C}_\alpha$ of size $n_\alpha$ with diameter $\lesssim 2D_\kphi$,
 \[
 \bx(0) \subset \bigcup_{\alpha} {\mathcal C}_\alpha, \quad \text{diam}({\mathcal C}_\alpha)  \lesssim 2D_\kphi, \quad n_\alpha:=\#\{j \ | \ \bx_j(0)\in {\mathcal C}_\alpha\}.
 \]
Assume  these clusters $\{{\mathcal C}_\alpha\}$ are bounded away from vacuum in the sense that $\nicefrac{n_\alpha}{N}\geq \theta >0$ (and thus, in particular, there are at most $1/\theta$ of them). Then  the flocking threshold depends solely on  the finite number of clusters, $\kappa_0 \lesssim \theta^{-4}$,   but otherwise  it does \emph{not} grow with $N$.
\end{corollary}

The  proof of theorems \ref{flockcont} (and likewise \ref{flockpart})  is argued  as follows. Let $u_+(t)$ and $u_-(t)$ denote (scalar) components of the velocity field $\bu(t)$ which achieve maximal, and respectively, minimal values at time $t$. Then $u_\pm(t)$ will be attracted  to align with their respective local averages, say $\bar{u}_\pm(t)$, localized to the corresponding $D_\kphi$-neighborhoods of $u_\pm(t)$. Thus, the diameter $D_{u(t)}:=|u_+(t)-u_-(t)|$ is expected to decay. However, the decay rate of $D_{u(t)}$  need \emph{not} be strictly negative: as noted in \cite[theorem 3.2]{ST2018}, there is a possibility of near constancy (flattening) of $u$ in the neighborhoods of the extrema $u_\pm(t)$. In this case, one needs to extend the first two neighborhoods  into a second set of neighborhoods of averages and so, thus creating a \emph{global} chain of balls of local neighborhoods, \emph{connecting} $\bar{u}_+(t)$ with $\bar{u}_-(t)$. This is outlined in section \ref{prelim}. In proposition \ref{endec}  we show that the \emph{global} variation over such chains is  controlled by   exponential decay of the  fluctuations of the velocity $|\delta \bu(t)|_2$, which in turn implies, proposition \ref{veldec},  exponential decay of the uniform  fluctuations $|\delta \bu(t)|_\infty$. The key question is whether \emph{chain connectivity} at the   local communication scale $D_\kphi$ persists over time, to facilitate  the above argument.
To this end we show, in proposition \ref{tdtp}, that  large enough $\kappa$  guarantees that the chain property at scale $\sim D_\kphi$ persists in time, and hence short-time process can be prolonged indefinitely thanks to the previously obtained local decay of $|\delta \bu(t)|_\infty$.
 
\section{Staying uniformly away from vacuum}
 As noted above, theorem \ref{flockcont} does \emph{not} rule out vacuum in small balls (with scale smaller than $\nicefrac{\rc}{100}$); in fact, all that we evolve are averaged masses.
It is known, however, that such \emph{uniform} bound  is intimatley related to the  existence of strong (smooth) solution; see e.g.,\cite{Ta2017,ST2017b}, and earlier results for non-uniform bounds in \cite{DKRT2018, ST2017a}. Here we show that our arguments can be adapted to obtain such a uniform  bound of the one-dimensional density $\rho$ away from $0$ in $\Omega=\bbt$.   We extend the techniques used in \cite{ST2017b} which cannot be directly used with short range  interaction kernels, proving  in section \ref{ubond} the following.
\begin{corollary}\label{boundro}
Let $(\rho,u)$ be a strong solution of the 1D system \eqref{system} on the torus $\bbt$ subject to non-vacuous data $(\rho_0\in L^1(\bbt), u_0\in C^1(\bbt)), \ \rho_0>0$, and satisfying the  assumption \ref{ass:s0}.  Then, there exists a positive constant $C_0$ such that
\begin{align}\label{eq:rmin}
\rho(t,\bx)\geq C_0>0 \qquad\mbox{for all}\ x\in{\mathbb T},\ t\geq 0.
\end{align}
\end{corollary}
 
The emergence of flocking for short-range interactions was treated in a  few previous works.  The first,  Shvydkoy-Tadmor \cite{ST2018}, in which they consider a short-range \emph{topological} interactions. In \cite{HKPZ2018},  S.-Y. Ha et. al. study the flocking of the \emph{one-dimensional } CS particle system \ref{cs} providing  explicit  critical coupling strength required for the emergence of flocking. 
C. Jin in \cite{Ji2018} treated the  Mostch-Tadmor agent-based system \cite{MT2011} with short-range communication, proving  flocking under small perturbation of regular-stochastic matrices.

\section{Chain connectivity}\label{prelim}

The global flocking analysis  of  system \eqref{system} and its underlying discrete version \eqref{cs} driven by  compactly supported  communication kernel $\kphi$ requires a localization procedure. To this end we note that the key local scale  in the dynamics of these systems  is the \emph{communication range} $\nicefrac{1}{2}D_\kphi$, where we shall often utilize the following basic observation, namely --- if two  sets $A, B\subset \Omega$ overlap a non-empty intersection, each with diameter $\text{diam}(A), \text{diam}(B)\leq \nicefrac{1}{2}D_\kphi$, then their entries are in direct communication so that $\displaystyle \mathop{\inf}_{\bx\in A,\, \by\in B}\phi(|\bx-\by|)\geq 1 $. We then perform  a localization procedure based in chain connectivity  of the initial support, $\Sp_0$.

\begin{definition}[{\bf Chain connectivity}]\label{cc} Fix $r>0$. A bounded domain ${\mathcal D}$ is  said to be \emph{chain connected at scale} $r$ if there exists a finite $\ind$ (possibly dependent on $r$) such that  every pair $\bx,\by \in \overline{{\mathcal D}}$ is connected through a chain ${\mathcal B}_{\bx,\by}$, of  at most $\ind$  balls of radii $r$ with non-empty overlap: 
\be\label{eq:h3}
{\mathcal B}_{\bx,\by}:=\{B_{r}(\bc_\alpha)\}_{\alpha=1}^{\ind} \ \ \text{such that} \ \ 
\left\{\begin{array}{l}
 B_{r}(\bc_{\alpha})\cap B_{r}(\bc_{\alpha+1})\neq \emptyset, \quad \bc_\alpha \in{\mathcal D}\vspace*{0.2cm}\\ 
\text{connecting} \ \bx\in \overline{B}_r(\bc_1)$ \ \text{and} \  $\by\in \overline{B}_r(\bc_\ind).
\end{array}\right.
\ee
\end{definition}

Note that  the balls $B_r(\bc_\alpha)$ are not requited to be inside ${\mathcal D}$ --- only their centers do, so in the sequel we can also treat discrete domains. 
 A regular, simply-connected domain with a smooth boundary is chain connected. For example, if the boundary of ${\mathcal D}$ satisfies the interior ball condition with scale  $r$ (so that each $\bx\in \partial {\mathcal D}$ admits  a ball $B=B_r(\by_\bx)\subset {\mathcal D}$ such that $\bx\in\partial B$) and the interior of ${\mathcal D}$ is arc-connected, then it is chain connected with scale $r$.
We shall not dwell on precise characterization of chain-connected domains and simply make the following.

\begin{assume}[{\bf $\Sp_0$ is chain connected and non-vacuous}]\label{ass:s0}
There exists $r>0$ such that the initial support $\Sp_0=\text{supp}(\rho_0)$ satisfies the following.

\smallskip\noindent
 (i) It is chain connected  at scale $\rc$, namely,
 \begin{equation}\label{cdef}
  \mbox{there exist} \  n \ \mbox{balls of diameter} \ \rc \ \mbox{with}  \ \rc \leq \frac{1}{6}D_\kphi \  \mbox{such that} \ \eqref{eq:h3} \ \mbox{holds};
 \end{equation}
 \smallskip\noindent
  (ii) All balls with center in $\Sp_0$ with diameter $\frac{\rc}{100}$ have a positive mass
\begin{equation}\label{mdef}
m=\inf_{B\in {\mathcal B}_{100}}\rho(B) >0, \quad {\mathcal B}_{100}:=
\Big\{B_{\frac{\rc}{100}}(\bc) \, | \, \bc\in \Sp_0\Big\}.
\end{equation}
\end{assume} 

\begin{remark}[{\bf On the length of the chain}]\label{rem:length} Since we did not insist on a minimal chain, its length $\ind$ (counting the number of balls) is allowed to be large. If we restrict attention to \emph{minimal} chain, however, then we may assume that each new step of the chain contains a non-negligible increment, say a ball of radius $\frac{\rc}{100}$
\[
B_{\rc}(\bc_{\alpha}) \setminus\bigcup_{\beta<\alpha} B_{\rc}(\bc_\beta) \supset B_{\frac{\rc}{100}}(\bb_\alpha) \  \mbox{centered around some} \ \bb_\alpha \in B_{\rc}(c_{\alpha}).
\]
Then   since $\big\{ B_{\rc}(\bb_\alpha)\big\}_\alpha$ are disjoint and recalling $\rho(\Sp_0)=1$, we conclude, in view of \eqref{mdef}, that $\displaystyle n(\rc) \lesssim \frac{1}{m}$. 
Alternatively, the cardinality  of  (minimal) chains  depends on the geometry of $\Sp_0$: for example,  if $\Sp_0$ is chord-connected then for `small' balls, $\displaystyle n(\rc) \lesssim \frac{D_{\Sp_0}}{\rc}$.
In particular, it follows that each `small' ball, say, $B_{\frac{\rc}{100}}(\cdot)$, is covered by at most a  \emph{fixed} ($\rc$-independent) number of balls from any chain ${\mathcal B}_{\bx,\by}$.
\end{remark}

Observe that assumption \ref{ass:s0} does \emph{not} require $\rho_0$ to be uniformly non-vacuous --- only its average over balls above scale $\frac{\rc}{100}$ need be bounded away from zero (of course, one can replace 100 by another fixed scale$>2$). Consequently, the chain connectivity also applies  in the drastically different scenario of $\rho_0$ being a sum of Dirac's masses, which enables us to treat the flocking behavior of both --- the discrete agent-based  dynamics \eqref{cs}  and the continuum hydrodynamics, treated  in theorems \ref{flockcont} and respectively \ref{flockpart} above.

\medskip\noindent
\paragraph{{\bf Propagation of connectivity}} Let us introduce the key component of our local-to-global approach, in which we  propagate the chain connectivity of $\Sp_0$  for the chain connectivity of $\Sp(t)$. 
Let $\XX^t(\bx)=\XX(t,\bx)$ denote the (forward) flow map, 
\begin{equation}\label{chardef}
\ddt\XX(t,\bx)=\bu(t,\XX(t,\bx)),\quad\mbox{and}\quad \XX(0,\bx)=\bx.
\end{equation}
It is well known that any smooth solution of the continuity equation preserves the mass of sets along the characteristic flow. That is
$ \rho(t,\XX^t(A))=\rho_{0}(A)$ for any set $A\subset \Sp_0$ with image $\XX^t(A)$ along the flow at time $t$. 
We claim that local propagation of mass ensures that the initial connectivity of $\Sp_0$ is  preserved in time. To this end, for any  $\bx,\by \in \overline{\Sp}(t)$, we let $\bx_0=\,^{t}\XX(\bx)\in \Sp_0$ and 
$\by_0=\,^{t}\XX(\by)\in \Sp_0$ denote the \emph{pre-images} of $\bx$ and $\by$, determined by the (backward) flow, abbreviated here as $^{t}\XX(\cdot)$.  Then by assumption, $\bx_0$ and $\by_0$  are connected through a chain, ${\mathcal B}_{\bx_0,\by_0}=\{B_r(\bc_\alpha)\}_\alpha$ of scale $\rc$. Let ${\mathcal F}_{\bx,\by}^t$ be the image chain by the forward flow
\be\label{eq:h4}
{\mathcal F}_{\bx,\by}^t=\{F_\alpha^t\}_{\alpha=1}^\ind, \qquad F_\alpha^t:= \XX^t(B_r(\bc_\alpha)), \quad \alpha=1,2, \ldots, \ind.
\ee
In this manner we obtain a finite chain ${\mathcal F}_{\bx,\by}^t$ connecting
$\bx\in F_1^t$ to $\by\in F_\ind^t$ with intersecting  consecutive pairs,
${F}^t_{\alpha}\cap {F}^t_{\alpha+1}\neq \emptyset$.

Thus,  the chains ${\mathcal F}_{\bx,\by}^t$ in \eqref{eq:h4}  induce  the chain connectivity of $\Sp(t)$,  in analogy to the connectivity of $\Sp_0$ by  the chains of balls ${\mathcal B}_{\bx_0,\by_0}$   in \eqref{eq:h3}.  The essential question is whether the \emph{scale} of the chain ${\mathcal F}_{\bx,\by}^t$, measured by 
$\displaystyle \mathop{\max}_\alpha \text{diam}(F_\alpha^t)$, retains the  order of the original scale  $\rc$ of ${\mathcal B}_{\bx_0,\by_0}$. It is here that  propagation of alignment properties of the flow play a central role.

\section{Proof of the main results}\label{main}
The proof of theorems \ref{flockcont} and \ref{flockpart} relies on three propositions. In the first, proposition \ref{endec}, we estimate a localized version of velocity fluctuations $|\delta \bu|_2$ and use the chain connectivity to piece together the local estimates and obtain local-in-time exponential decay of whole $|\delta \bu|_2$. Then, in Proposition \ref{veldec}, we use this decay to prove local-in-time exponential decay of $|\delta \bu|_\infty$. We need an $L^\infty$ decay because, subsequently in Proposition \ref{tdtp}, we bound the total distance traveled and show that with $\kappa$ large enough, the particles do not escape the range dictated by the essential support of $\kphi$. This enables us to prolong the local-in-time estimates and prove the main results.

\paragraph{{\bf Exponential energy decay}}\label{mainestim}
\noindent To prove the exponential decay of the kinetic energy, we introduce the maximal time, ${\mathcal T}=\sup \tau(\kappa)$, during which the flow map does not expand more than $3\rc$,
\begin{align}\label{tal}
\mathcal{T}:=\sup\Bigg\{\tau>0: \sup_{\substack {\bx_0,\by_0\in \Sp_0 \\ |\bx_0-\by_0|\leq r} }|\XX^{\tau}(\bx_0)-\XX^{\tau}(\by_0)|\leq 3\rc\Bigg\}.
\end{align}
Since by assumption, \eqref{system} admits a smooth solution, there exists a small enough  ${\mathcal T}>0$ such that \eqref{tal} holds. We will show in Proposition \ref{tdtp} below, that during that time interval $[0,\mathcal{T}]$, the fluctuations $|\delta \bu(t)|_2$ decay exponentially. Our objective is then to show  that $\mathcal{T}=\infty$ when $\kappa$ is large enough. We achieve this in proposition \ref{tdtp}.

\begin{proposition}\label{endec}
Let $(\rho, \bu)$ be a strong solution to \eqref{system} in $[0,{\mathcal T}]$ subject to initial data $(\rho_0,\bu_0)\in (L^1(\Sp_0), C^1(\Sp_0))$ satisfying assumption \ref{ass:s0}. In particular, $\text{supp}(\rho_0)$ is chain connected at scale $\rc\leq \frac{1}{6}D_\kphi$.
Then, there exits a constant $\eta=\eta(\rc,m)>0$,  such that
\begin{equation}\label{eq:ener}
|\delta \bu(t)|_2\leq e^{-\kappa \eta t}|\delta \bu_0|\hspace{1em}\mbox{for all}\ \ t\leq \mathcal{T}, \qquad \eta\lesssim \frac{m^2}{n(\rc)}.
\end{equation}
\end{proposition}
\noindent
Remark that in proposition \ref{tdtp} below we derive a minimal value of the threshold  $\kappa$ which is inversely proportional to $\eta$, hence we will pay attention to tracing upper-bound of the latter in terms of $m$ and $n(\rc)$.
\begin{proof}
We begin  by noting that during the  time interval $[0,{\mathcal T}]$, the maximal diameter of each element in the connecting chain $F_\alpha^t\in {\mathcal F}_{\bx,\by}^t$ cannot expand more than $\rc$ units beyond the diameter of its pre-image ball $B_{\rc}(\bc_\alpha)$,
\begin{align}\label{S2}
{\rm diam}(F_\alpha^t)\leq 3\rc,\qquad  t\leq{\mathcal T}.
\end{align}
Moreover, $\rho(F_\alpha^t)=\rho(B_{\rc}(\bc_\alpha))\geq m$.
The proof proceeds in three steps.

\medskip\noindent
{\bf Step 1} (Restriction). We begin by making a general observation that allows us to restrict attention to fluctuations over sets with diameter width $6\rc\leq D_\kphi$. 
Specifically, we  claim that for \eqref{eq:ener} to hold, it suffices to bound  energy fluctuations in terms of the restricted enstrophy on stripes of width$\leq 6\rc$
\begin{equation}\label{ineq1}
|\delta \bu(t)|_2^{2} \leq \frac{1}{\eta} \JR, \quad \JR:= \int_{\bx\in\Sp}\int_{\by\in\Sp\cap B_{6\rc}(\bx)}\hspace*{-0.5cm}|\bu(t,\bx)-\bu(t,\by)|^{2} \drho(\bx)\drho(\by).
\end{equation}
Indeed,  since $6\rc$ does not exceed the communication scale $D_\kphi$ so that $\phi(|\bx-\by|)\geq 1$ wherever $\by\in \Spt\cap B_{6\rc}(\bx)$, the enstrophy bound \eqref{ens}  yields  
\begin{align*}\begin{aligned}
\frac{1}{2}\ddt |\delta \bu(t)|^2_2 &= -\kappa\int_{\Spt\times\Spt}\hspace*{-0.8cm} \kphi(|\bx-\by|) \big|\bu(t,\bx)-\bu(t,\by)\big|^{2} \drho(\bx) \drho(\by)\\ 
& \leq- \kappa  \int_{\bx\in\Sp}\int_{\by\in\Spt\cap B_{6\rc}(\bx)}\hspace*{-1.3cm}\kphi(|\bx-\by|)|\bu(t,\bx)-\bu(t,\by)|^{2} \drho(\bx)\drho(\by)\leq -\kappa \JR, 
\end{aligned}
\end{align*}
 and together with the assumed bound \eqref{ineq1}, we end up with exponential decay rate corresponding to \eqref{eq:ener}
\[
\ddt|\delta \bu(t)|_2^{2} \leq -2\kappa\JR 
         \leq -2\kappa \eta |\delta \bu(t)|_2^{2}.
\]

\medskip\noindent
{\bf Step 2} (Localization). To simplify notation  we often suppress the time dependence of quantities  evaluated at time $t$, abbreviating $\Sp=\Spt, \bu(\bx)=\bu(t,\bx), F_\alpha=F_\alpha^t$ etc. In the sequel we shall work with averages rather than point values: we use $\bar{\bu}_A$ to denote the average 
$\displaystyle \bar{\bu}_A:=\frac{1}{\rho(A)}\int_A \bu(\bx)\drho(\bx)$.\newline
  For an arbitrary pair  $\bx,\by\in \overline{\Sp}$, we consider their connecting chain ${\mathcal F}_{\bx,\by}^t=\{F_\alpha\}_{\alpha=1}^\ind$, and express the difference $\bu(\bx)-\bu(\by)$ in terms of the intermediate averages along that chain,
\[
\bu(\bx)-\bu(\by)\equiv (\bu(\bx)-\bar{\bu}_{1}) + \sum_{\alpha=1}^{\ind-1} (\bar{\bu}_{{\alpha}}-\bar{\bu}_{{\alpha+1}}) + (\bar{\bu}_{\ind}-\bu(\by)), \quad \bar{\bu}_{\alpha}:= \bar{\bu}_{F_\alpha}.
\]
 This enables us to decompose   a global energy fluctuations bound into  
\begin{align*}
\int_{\Sp\times\Sp}&|\bu(\bx)-\bu(\by)|^2 \drho(\bx)\drho(\by) \\
&  \leq (\ind+1) \Big(\int_{\by\in \Sp}\int_{\bx\in \Sp} \hspace*{-0.3cm}|\bu(\bx)-\bar{\bu}_{1}|^2 \drho(\bx)\drho(\by) \\
& \qquad \qquad  \ \ + \int_{\Sp\times \Sp} \hspace*{-0.0cm} \sum_{\alpha=1}^{\ind-1} |\bar{\bu}_{{\alpha}}-\bar{\bu}_{{\alpha+1}}|^2\drho(\bx)\drho(\by)
+\int_{\bx\in \Sp}\int_{\by\in \Sp}\hspace*{-0.3cm} |\bu(\by)-\bar{\bu}_{\ind}|^2 \drho(\by)\drho(\bx)\Big) 
\\
&  =:(n+1)( I + II + III).
\end{align*}
Note that the averages $\{\bar{\bu}_\alpha\}_\alpha$   are dependent on $(\bx,\by)$, 
and in particular therefore,  the first and third integrands on the right depend on $(\bx,\by)$ because $\bar{\bu}_1$ and $\bar{\bu}_n$ do.
It remains to bound the local fluctuations on the right.

\medskip\noindent
{\bf Step 3} (Local estimates). 
  We  bound the first term first.
Since $F_1 \ni \bx$ and since by \eqref{S2} its diameter, at least for $ t\leq {\mathcal T}$, does not exceed  $3\rc$, which is less than the communication scale $\nicefrac{1}{2}D_\kphi$, we have
\begin{align*}
\int_{\bx\in \Sp} \hspace*{-0.3cm}|\bu(\bx)-\bar{\bu}_{1}|^2 \drho(\bx)
&= \int_{\bx\in \Sp} \hspace*{0.1cm}\frac{1}{\rho^2(F_1)}\left|\int_{\by'\in F_1}\big(\bu(\bx)-\bu(\by')\big) \drho(\by')\right|^2 \drho(\bx)\\
& \leq \int_{\bx\in \Sp} \hspace*{0.1cm}\frac{1}{\rho(F_1)} \int_{\by'\in \Sp\cap B_{3\rc}(\bx)}\hspace*{-0.3cm} |\bu(\bx)-\bu(\by')|^2 \drho(\by') \drho(\bx)\leq \frac{1}{m} \JR.
\end{align*}
The last bound is uniform with respect to the dependence of $\bar{\bu}_1$ on $\by$, and  since $\Sp$ has unit mass, 
\[
I=\int_{\by\in \Sp}\int_{\bx\in \Sp} \hspace*{-0.3cm}|\bu(\bx)-\bar{\bu}_{1}|^2 \drho(\bx)\drho(\by) \leq \frac{1}{m} \JR.
\]
Similarly,  for the third term, $\displaystyle III=\int_{\bx\in \Sp}\int_{\by\in \Sp}\hspace*{-0.3cm} |\bu(\by)-\bar{\bu}_{\ind}|^2 \drho(\by) \leq \frac{1}{m}\JR$.

\smallskip\noindent
To estimate the difference of consecutive averages in $II$,  we begin by noting that 
\be\label{eq:basic}
\rho(A)\rho(B) |\bar{\bu}_A-\bar{\bu}_B|^2 \leq \int_{(\bx',\by')\in (A\times B)}\hspace*{-0.3cm}|\bu(\bx')-\bu(\by')|^2 \drho(\bx')\drho(\by'),
\ee
which is nothing but the Cauchy-Schwartz bound for the identity  
\[
\bar{\bu}_A-\bar{\bu}_B=\int_{(\bx',\by')\in (A\times B)}\hspace*{-0.3cm}\big(\bu(\bx')-\bu(\by')\big)\frac{\drho(\bx')\drho(\by')}{\rho(A)\rho(B)}
\]
Set  $A=F_{\alpha}^t$ and  $B=F_{\alpha+1}^t$, and recall that the mass of each is at least $m$; hence
\be\label{eq:sub}
|\bar{\bu}_{\alpha}-\bar{\bu}_{\alpha+1}|^2\leq \frac{1}{m^2}\int_{\bx'\in F_{\alpha+1}}\int_{\by'\in F_{\alpha}} |\bu(\bx')-\bu(\by')|^2 \drho(\by)\drho(\bx).
\ee
By assumption, $F_\alpha$ and $F_{\alpha+1}$ (which  still  depend on $(\bx,\by)$) have a non-empty intersection hence \eqref{S2} implies for $(\bx',\by')\in (F^t_{\alpha+1}\times F^t_\alpha) \leadsto |\by'-\bx'|\leq \text{diam}(F^t_{\alpha+1})+\text{diam}(F^t_\alpha) \leq 6\rc$,
\[
|\bar{\bu}_{\alpha}-\bar{\bu}_{\alpha+1}|^2\leq \frac{1}{m^2}\int_{\bx'\in F_{\alpha+1}}\int_{\by'\in \Sp\cap B_{6\rc}(\bx')} \hspace*{-0.3cm}|\bu(\bx')-\bu(\by')|^2 \drho(\by')\drho(\bx').
\]
We note that although the elements in the chain ${\mathcal F}^t_{\bx,\by}$ have non-empty overlaps, the assumed minimality of the chain implies that the union   $\cup_\alpha F_\alpha^t$ can cover $\Spt$ only  \emph{finitely} many times (see remark \ref{rem:length} above). Consequently, 
the second term does not exceed  (recall $ \rho(\Spt)\equiv 1$) 
\begin{align*}
II &= \int_{\Sp\times \Sp}\sum_{\alpha=1}^{\ind-1}|\bar{\bu}_{\alpha}-\bar{\bu}_{\alpha+1}|^2\drho(\bx)\drho(\by)\\
& \lesssim \frac{1}{m^2}\int_{\bx'\in \Sp}\int_{\by'\in \Sp\cap B_{6\rc}(\bx')}\hspace*{-0.3cm} |\bu(\bx')-\bu(\by')|^2 \drho(\by')\drho(\bx') \leq \frac{1}{m^2} \JR.
\end{align*}
We conclude that \eqref{ineq1} holds with $\displaystyle \eta \lesssim \frac{m^2}{n(\rc)}$:
\begin{align}\label{mu}
|\delta \bu(t)|_2^2 \lesssim (n+1)\big(\frac{2}{m}+\frac{1}{m^2}\big) \JR, 
\end{align}
and \eqref{eq:ener} follows.
\end{proof}

\paragraph{{\bf Bounding the diameter of velocity field}} 
 The next proposition provides an estimate of the maximal velocity fluctuations, $|\delta \bu|_\infty$, in terms of the kinetic energy and coefficients that decay exponentially.
\begin{proposition}\label{veldec}
Let $(\rho, \bu)$ be a strong solution to \eqref{system} in $[0,{\mathcal T}]$ subject to initial data $(\rho_0,\bu_0)\in (L^1(\Sp_0), C^1(\Sp_0))$ with a chain connected support $\text{supp}(\rho_0)$ of scale $\rc\leq \frac{1}{6}D_\kphi$.
Then we have
\begin{align}\label{veldec0}
|\delta \bu(t)|_{\infty}\leq |\delta \bu_{0}|_{\infty}e^{-\kappa m t} + 
\kappa \int_{0}^{t}e^{-\kappa  m(t-\sigma)}|\delta \bu(\sigma)|_2 d\sigma, \qquad \ t \leq T.
\end{align}
\end{proposition}

\begin{proof}
We first employ the argument of \cite[theorem 2.3]{MT2014}, \cite[Sec. 1]{HeT2017} in order to restrict attention to a scalar projection of $\bu$: fix an arbitrary  unit vector $\bw$, then a projection of \eqref{system} on $\bw$ yields
\[
\partial_t \langle \bu,\bw\rangle + \bu\cdot\nabla_\bx \langle \bu,\bw\rangle = \kappa
\int_\Spt \phi(|\bx-\by|)(\langle \bu(\by),\bw\rangle-\langle \bu(\bx),\bw\rangle)\drho(\by).
\]
Let $\bu_\pm(t)$ be the extremal projected values, $\langle \bu_\pm,\bw\rangle =\left\{\begin{array}{c}{\max}_\Spt\\ {\min}_\Spt\end{array}\right\}  \langle \bu,\bw\rangle$ at $\bx_+=\bx_+(t)$ and, respectively, at $\bx_-=\bx_-(t)$\footnote{$\bx_\pm$ need not be unique --- \emph{any} extreme location will suffice.}. Let $F_+^t\ni \bx_+$ and $F_-^t\ni \bx_-$ be the two sets from the chain ${\mathcal F}^t$ covering $\Sp(t)$. Since $\text{diam}({\mathcal F}_\pm^t) \leq 3\rc \leq \nicefrac{1}{2}D_\kphi$,
\begin{align*}
\ddt \langle \bu_+,\bw\rangle  &= \kappa
\int_\Spt \phi(|\bx_+-\by|)(\langle \bu(\by),\bw\rangle-\langle \bu_+,\bw\rangle)\drho(\by)
\leq \kappa
\int_{F_+^t}(\langle \bu(\by),\bw\rangle-\langle \bu_+,\bw\rangle)\drho(\by),\\
\ddt \langle \bu_-,\bw\rangle  & = \kappa
\int_\Spt \phi(|\bx_--\by|)(\langle \bu(\by),\bw\rangle-\langle \bu_-,\bw\rangle)\drho(\by)
\geq \kappa 
\int_{F_-^t} (\langle \bu(\by),\bw\rangle-\langle \bu_-,\bw\rangle)\drho(\by),
\end{align*}
and since the mass of $F_\pm^t$ is at least $m$, we control the  velocity \emph{diameter} ${D}_{\bu(t)}:=\bu_+(t)-\bu_-(t)$ projected  on $\bw$,
\begin{align}\label{eq:constancy}
\ddt &\langle {D}_{\bu(t)}, \bw\rangle\nonumber \\
&  \leq  \kappa
\frac{m}{\rho(F_+^t)}\int_{F_+^t}\hspace*{-0.2cm}(\langle \bu(\by),\bw\rangle-\langle \bu_+,\bw\rangle)\drho(\by)
-\kappa\frac{m}{\rho(F_-^t)}
\int_{F_-^t}\hspace*{-0.2cm} (\langle \bu(\by),\bw\rangle-\langle \bu_-,\bw\rangle)\drho(\by)\\
& =  -\kappa  m \langle {D}_{\bu(t)},\bw\rangle + \kappa   m \langle \bar{\bu}_{F_+^t} -\bar{\bu}_{F_-^t},\bw\rangle, \qquad \langle {D}_{\bu(t)},\bw\rangle=\langle \bu_+(t)-\bu_-(t),\bw\rangle.
\nonumber
\end{align}
By \eqref{eq:basic}, $m|\bar{\bu}_{F_+^t} - \bar{\bu}_{F_-^t}|$ does not exceed 
$|\delta \bu|_2$,  hence
\[
\ddt \langle {D}_{\bu(t)}, \bw\rangle \leq -\kappa  m \langle {D}_{\bu(t)},\bw\rangle 
+ \kappa  |\delta \bu(t)|_2.
\]
Integration yields 
$\langle {D}_{\bu(t)}, \bw\rangle \leq \langle {D}_{\bu_{0}}, \bw\rangle e^{-\kappa m t} + 
\kappa \int_{0}^{t}e^{-\kappa  m(t-\sigma)}|\delta \bu(\sigma)|_2 d\sigma$,
and  taking the supremum over all unit $\bw$'s recovers the desired bound \eqref{veldec0} for $|\delta \bu(t)|_\infty = \displaystyle \mathop{\max}_{|\bw|=1}\langle {D}_{\bu(t)}, \bw\rangle$.
\end{proof}

\begin{remark}
Observe that the decay rate for the diameter $\langle D_{\bu(t)},\bw\rangle$ in \eqref{eq:constancy} need \emph{not} be strictly negative: as noted in \cite[theorem 3.2]{ST2018}, there is a possibility of constancy (flattening) of $\bu(t)$ in the $F_\pm^t$-neighborhoods of the extrema $\bu_\pm(t)$, where the RHS of \eqref{eq:constancy} cannot be bounded away from $0-$. Instead, this is circumvented here by plugging  the known exponential decay of the energy \eqref{eq:ener} into \eqref{veldec0} to conclude
\begin{align}\label{eq:delbd}
|\delta \bu(t)|_\infty & \leq |\delta \bu_0|_\infty e^{-\kappa  mt} 
+ \kappa  \cdot |\delta \bu_0|_2 \cdot e^{-\kappa  mt} \int_0^t e^{\kappa  (m-\eta)\sigma}d\sigma   \nonumber \\
& \leq |\delta \bu_0|_\infty \cdot e^{-\kappa  mt}  +  K \cdot |\delta \bu_0|_2\cdot e^{-\kappa  \eta t}, \qquad K:=\frac{1}{m-\eta} \leq \frac{2}{m}. 
\end{align}
\end{remark}

\paragraph{{\bf Bound on the total distance traveled}}\label{tdt}
Here we show that with large enough $\kappa$, the time interval  $[0,\mathcal{T}]$ used in Propositions \ref{endec} and \ref{veldec} can be extended indefinitely.

\begin{proposition}[Proof of theorem \ref{flockcont}]\label{tdtp}
Let $(\rho, \bu)$ be a strong solution to \eqref{system} in $[0,{\mathcal T}]$ subject to initial data $(\rho_0,\bu_0)\in (L^1(\Sp_0), C^1(\Sp_0))$ satisfying assumption \ref{ass:s0}. In particular, the $\text{supp}(\rho_0)$ is assumed to be chain connected at scale $\rc\leq \frac{1}{6}D_\kphi$.
Further, assume that the amplitude of alignment is large enough
\begin{align}\label{kll}
\kappa\geq \kappa_0:=\frac{1}{\rc m}\Big(\frac{1}{2}|\delta \bu_{0}|_{\infty}+\frac{1}{\eta} |\delta \bu_0|_2\Big), \quad \eta \lesssim \frac{m^2}{n(\rc)}.
\end{align}
Then the chain connectivity of the associated flow map in \eqref{tal} persists  for all time, ${\mathcal T}=[0,\infty)$ and theorem \ref{flockcont} with \eqref{eqs:fc} follows.
\end{proposition}
Indeed,  integrating the fluctuations of \eqref{chardef}, $\ddt \delta \XX^t(\cdot)=\delta\bu(t,\cdot)$, and utilizing \eqref{eq:delbd} we find the uniform-in-$\tau$ bound
\[
|\XX^{\tau}(\bx_0)-\XX^{\tau}(\by_0)|  \leq |\bx_0-\by_0|+\int_{0}^{\tau}|\delta \bu(t)|_{\infty}\hspace{1mm}\dt
< |\bx_0-\by_0|+ \frac{|\delta \bu_{0}|_{\infty}}{\kappa  m} +  \frac{2}{m}\cdot \frac{|\delta \bu_0|_2}{\kappa  \eta}, \quad \tau<\infty.
\]
Thus, large enough  $\kappa$, \eqref{kll}, guarantees that initial fluctuations of size $|\bx_0-\by_0|\leq \rc$ do not expand more than $ |\XX^{\tau(\kappa)}(\bx_0)-\XX^{\tau(\kappa)}(\by_0)|\leq 3\rc$ for \emph{all} $\tau$'s, and hence the maximal time of connectivity 
${\mathcal T}(\kappa)\!:=\sup{\tau(\kappa})\!=\!\infty$, proving the exponential flocking asserted in theorem \ref{flockcont}.

\begin{remark}[{\bf Large communication scale}]
Observe that if the essential support of  $\kphi$ is large, $D_\kphi \gg 1$, then $\kappa_0$ with $\rc \gg 1$ (recall $m\gtrsim 1$ and hence $n(\rc) \lesssim 1$)  recovers \emph{unconditional flocking} provided
\[
\mathop{\min}_{h\leq \rc}\kphi(h)\rc \ \stackrel{\rc\rightarrow \infty}{\longrightarrow} \ \infty,
\]
which is only slightly stronger than the `fat tail' condition \eqref{eq:flocking}.
\end{remark}

\paragraph{{\bf Proof of the discrete case (theorem \ref{flockpart})}}\label{mainproof}
We consider the discrete  solution $\rho_0=\frac{1}{N}\sum_{j=1}^N \delta_{\vx_j(0)}(\vx)$, with the discrete support $\Sp_0=\{\vx_j(0)\}_{j=1}^N$, normalized to have a unit mass so that \eqref{eq:mass} holds.
The assumed chain connectivity of the discrete $\Sp_0$ involves no more than $n=N$ balls with diameters $\rc$, connecting every initial pair $(\vx_i(0), \vx_j(0))$. The average mass of each $r$-ball centered $\Sp_0$ --- that is, at one of the $\vx_j(0)$'s has mass $\geq m=\frac{1}{N}$ so that we maintain the counting bound $nm\lesssim 1$.
With this we recover the bound \eqref{kll} with  $\eta\lesssim \nicefrac{1}{N^3}$ (and since $N|\delta \vv_0|_2 > |\delta \vv_0|_\infty$)
\be\label{dll}
 \kappa_0\lesssim \frac{N^4}{\rc}|\delta \vv_0|_2.
\ee

\section{Uniformly bounded density  away from vacuum}\label{ubond}
The density in \eqref{system} plays a central role as the 
carrier of local averaging, and hence a uniform bound on the density $\rho(t,\cdot)$ away from vacuum is essential for the existence of strong solutions to \eqref{system} and their asymptotic behavior. This is particularly relevant in the case of \emph{singular interaction kernels}, \cite{ST2017a,ST2017b,ST2017c,DKRT2018, ST2018}. As noted in \cite{ST2018}, the lower bound $\rho\gtrsim \nicefrac{1}{\sqrt{1+t}}$ will suffice to yield unconditional flocking in the general multiD case.  In \cite{ST2017a, DKRT2018,ST2018} it was shown that the 1D density is controlled from below, $\rho\geq  1/(1+t)$, while in \cite{ST2017b} this estimate was improved to a uniform bound $\rho\geq c_0>0$. In all cases, the  bounds depend on the long-range support of $\kphi$. Here we treat the case of short-range kernels, proving the uniform bound away from vacuum  by adapting our  chain connectivity to the proof \cite[Lemma 3.1]{ST2017b}.

\begin{proof}[Proof of Corollary \ref{boundro}]
As in \cite{ST2017b}, we use the quantity $q$ 
\[
q(t,x):=\frac{1}{\rho(t,y)}\left(u_x(t,x) + \kappa\int_{y\in \bbt} \kphi(|x-y|)(\rho(t,y)-\rho(t,x))dy\right)
\]
to rewrite the mass equation \eqref{system}$_1$ in an equivalent   diffusive form 
(abbreviating $\rho(t,\cdot)=\rho(\cdot)$)
\begin{align*}
\rho_t(x) + u\rho_x(x) = -q\rho^2(x) + \kappa\rho(x)\int_{\mathbb T}\kphi(|x-y|)(\rho(y)-\rho(y))\dy.
\end{align*}
Evaluating the above expression at $\rho_-(t)=\min_{x\in{\mathbb T}}\rho(t,x)$ leads to
\begin{align*}
\ddt\rho_- = -q\rho_-^2 + \kappa\rho_-\int_{\mathbb T}\kphi(|x_- - y|)(\rho(y)-\rho_-)\dy \geq -q\rho_-^2 + \kappa\underbrace{\rho_-\int_{F^t_-}\kphi(|x_- - y|)(\rho(y)-\rho_-)\dy}_{\geq 0}.
\end{align*}
Here $x_-=x_-(t)$ is any point at which $\rho(t)$ attains its minimum, and $F_-^t\in {\mathcal F}$ is any local neighborhood surrounding   $x_-\in F_-^t$. 
Recall that according to \eqref{S2}, the diameter ${\rm diam}(F_-^t)\leq 3r <\nicefrac{1}{2}D_\kphi$ and hence $\kphi(|x_--y|)\geq 1$ for $y\in F^t_-$. Thus, since $\rho(F_-^t) \geq m$,
\begin{align*}
\ddt\rho_- \geq -q\rho_-^2 + \kappa\rho_-\cdot\int_{ F_-^t}\kphi(|x_--y|)(\rho(\by)-\rho(x_-))\dy
\geq  -q\rho_-^2 +\kappa \rho_- \cdot m - \rho_-^2\kappa |F_-^t|,
\end{align*}
where the volume of $F_-^t$ is bounded since its diameter is. We end up with
\begin{align*}
\ddt\rho_- &\gtrsim  -(|q|_\infty+\kappa) \rho_-^2 + \kappa  m \cdot \rho_-.
\end{align*}
What signifies the quantity $q$ is that it is being transported, $q_t+uq_x=0$, e.g., \cite[\S2.2]{ST2017a}, hence  $|q|_\infty \leq |q_0|_\infty$, and we conclude with a uniform lower-bound \eqref{eq:rmin}
\[
\min_{x\in \bbt}\rho(t,x)\gtrsim\min\Big\{\rho_{-}(0), \frac{\kappa m}{|q_0|_{\infty}+\kappa}\Big\} >0.
\]
\end{proof}

\section{Closing remarks}\label{remarks}
While we do not claim optimality of the threshold $\kappa_0$, we observe that  the existence of such threshold  is required: it is easy to construct a counterexample where the intensity of the interaction $\kappa$ is not sufficient to ensure flocking for given initial condition and $D_\kphi$. Indeed, let us consider a bounded communication weight satisfying \eqref{psi} with $D_\kphi=4$ and two particles on a line with $x_1(0)=-1$, $x_2(0)=1$, $v_1(0)=-1$, $v_2(0)=1$. Then, putting
\begin{align*}
x(t)=x_2(t)-x_1(t)>0,\qquad v(t)=v_2(t)-v_1(t)>0, 
\end{align*}
we reformulate system \eqref{cs} into
\begin{align}\label{2part}
\dot{x} =v,\qquad \dot{v} = -2\kappa v\kphi(x) = -2\kappa\ddt\kPhi(x), \qquad \Phi(r):=\int_0^r \phi(h)\dh.
\end{align}
Integration of \eqref{2part} leads to
\begin{align*}
v(t) &= 2\kappa\kPhi(x(0))-2\kappa\kPhi(x(t)) + v(0)\\
x(t) &= 2\kappa t\kPhi(x(0)) -2\kappa\int_0^t\kPhi(x(t))\dt +tv(0)+x(0).
\end{align*}
Since $\kphi$ is bounded and $0$ outside of, say, $[0,2D_\kphi]$ then $\kPhi$ is also bounded. Therefore, as long as $x\leq 2D_\kphi=8$ we have
\begin{align*}
x(t) &\geq -8\kappa t|\kPhi|_\infty +tv(0)+x(0).
\end{align*}
Now, if $\displaystyle \kappa\leq \nicefrac{v(0)}{16|\kPhi|_\infty}$ then
$x(t) \geq \frac{t}{2}v(0)+x(0)$ and $x$ grows steadily with the rate that can be bounded from below by $t v(0)/2$  and eventually reaches $2D_\kphi=8$. As soon as this happens $\kphi(x)=0,$ and the particles move away from each other with a constant velocity. A similar example can be produced for singular $\kphi$.

\end{document}